\documentclass[12pt]{amsart}
\usepackage[toc,page]{appendix}
\usepackage[english]{babel}
\usepackage{amsmath}
\usepackage{amssymb,amsthm,amscd}
\usepackage[pdftex]{color}
\usepackage[textwidth=18cm,textheight=24cm]{geometry}
\usepackage{mathrsfs}
\usepackage[thinlines]{easybmat}

\newcommand{\C}{\mathbb{C}}

\newcommand{\cH}{\mathcal{H}}

\newcommand{\R}{\mathbb{R}}

\newcommand{\N}{\mathbb{N}}

\newcommand{\I}{\mathbb{I}}
\renewcommand{\d}{\mathrm{d}}

\DeclareMathAlphabet{\mathpzc}{OT1}{pzc}{m}{it}

\newtheorem{pro}{Proposition}

\newtheorem{definition}{Definition}
\newtheorem{theorem}{Theorem}
\newtheorem{cor}{Corolary}

\begin{document}
\title{On the Christoffel--Darboux formula for\\ generalized  matrix orthogonal  polynomials of multigraded-Hankel type}
\author{Carlos \'{A}lvarez-Fern\'{a}ndez}
\address{Departamento de M\'{e}todos Cuantitativos, Universidad Pontificia Comillas, 28015-Madrid, Spain}
\email{calvarez@cee.upcomillas.es}
\author{ Manuel Ma\~{n}as}
\address{Departamento de F\'{\i}sica Te\'{o}rica II, Universidad Complutense de Madrid, 28040-Madrid, Spain}
\email{manuel.manas@ucm.es}
\thanks{MM thanks economical support from the Spanish ``Ministerio de Econom\'{\i}a y Competitividad" research project MTM2012-36732-C03-01,  \emph{Ortogonalidad y aproximacion; Teoria y Aplicaciones}}

\maketitle
\begin{abstract}
We obtain an extension of the Christoffel--Darboux formula for matrix orthogonal polynomials with a generalized Hankel symmetry, including the Adler-van Moerbeke generalized orthogonal polynomials.
\end{abstract}
\section{Introduction}

In \cite{adler-van-moerbeke} Adler and van Moerbeke considered generalized orthogonal polynomials which depended on  a band condition that we extended in \cite{cum} to a more general scenario: what we called multi-graded Hankel situation. This leads to a multicomponent extension of generalized orthogonal polynomials which in some cases can be described in terms of multiple orthogonal polynomials of mixed type with a type II normalization. The aim of this paper is to construct the Christoffel--Darboux formula in this context,  using the the Gaussian (or $LU$) factorization techniques.

One of the approaches to the classical integrable systems comes from the use of infinite dimensional Lie groups, that has its roots in the works of M. Sato \cite{sato}, M. Mulase \cite{mulase}, and the profound analysis of the Toda integrable hierarchy made by K. Takasaki and K. Ueno  \cite{ueno-takasaki}. Mark Adler and Pierre van Moerbeke performed an important  work connecting  the Gaussian factorization methods in the KP hierarchy with orthogonal polynomials, generalized orthogonality, and Darboux transformations, see \cite{adler}-\cite{adler-vanmoerbeke-5}. The last of them, signed also by P. Vanhaecke, studied multiple orthogonal polynomials, and subsequently multi-component hierarchies. The link that they use in \cite{adler-vanmoerbeke-5} is the Riemann--Hilbert approach that E. Daems and A. Kuijlaars had developed before  in \cite{daems-kuijlaars0} and \cite{daems-kuijlaars}. Those ideas of using Riemann--Hilbert techniques to orthogonal polynomials come from the seminal paper \cite{fokas} by A. S. Fokas, A. R. Its and A. V. Kitaev. We also mention
 \cite{manas-martinez-alvarez} and \cite{manas-martinez} where the authors revisited the group theoretic study of the so called multi-component 2D-Toda hierarchy using the known tools of Gaussian factorization.

Guided by the ideas of \cite{manas-martinez-alvarez}, \cite{bergvelt}, \cite{adler-vanmoerbeke-5} and \cite{Cafasso} we studied, using Gaussian factorization techniques,  the connections between 2D Toda-type integrable hierarchies  with matrix bi-orthogonal polynomials \cite{cum}, multiple orthogonal polynomials of mixed type \cite{afm-2} and orthogonal polynomials on the unit circle \cite{am-CMV}. An interesting byproduct  is a method -- based on the $LU$ factorization-- that turns to be useful when one is interested in generalizations of the Christoffel--Darboux formula for orthogonal polynomials. Despite, the formulas were known (e.g. \cite{daems-kuijlaars}, and \cite{Barroso-Vera}),  the method is general enough to be applied in many  different situations. In this paper we present the analysis for generalized matrix bi-orthogonal polynomials with generalized Hankel symmetries \cite{cum}.

The layout of the paper is as follows. After this introduction in \S \ref{section.intro} we explain the basics of the Gaussian factorization method for matrix bi-orthogonal polynomials in section \ref{section.intro.Gaussian}. The next step is to define the kind of symmetries that are present in this framework, that is what we do in section \ref{section.Hankel.multigraded}. To conclude we put all the pieces together in section \ref{section.Matrix.CD} with a theorem that shows the application of the factorization techniques to the obtention of generalized Christoffel--Darboux formulas.
\label{section.intro}
\subsection{Remainder: The Gaussian factorization method for matrix bi-orthogonal polynomials}
\label{section.intro.Gaussian}
Let us consider a family of $N \times N$ matrix bi-orthogonal polynomials on the real line $\{p_n,\tilde p_n\}$ and a matrix weight $\rho:\R\to \C^N\times \C^N$ such that for $n,n' =0,1,\dots$
\begin{align}\label{biorth}
\langle p_n, \tilde p_{n'} \rangle := \int_{\R} p_n(x)\rho(x) \tilde p_{n'}(x)^{\top} \d x=\delta_{n,n'} \I_N.
\end{align}

It follows from equation \eqref{biorth} that both families of polynomials satisfy orthogonality relations that can be written as
\begin{align}
\langle p_n, x^j \rangle &=\int_{\R}p_n(x) \rho(x) x^j \d x=0, \\
\langle x^j, \tilde p_n \rangle &= \int_{\R} x^j \rho(x) \tilde p_{n}(x)^{\top} \d x=0.
\end{align}

Given a weight that determines a matrix orthogonality problem, it is possible to define what we call the moment matrix of the weight
\begin{definition}
Given a matrix weight $\rho$ we can define the moment matrix $g$ as the following semi-infinite $N \times N$ block matrix
\begin{align}
g_{ij}&=\int_{\R}x^i \rho(x) x^j \d x=\int_{\R}\chi(x)\rho(x)\chi(x)^{\top} \d x,
\end{align}
where $\chi$ is the monomial matrix sequence $\chi(x):=(I_N, x I_N, x^2 I_N, \dots)^{\top}$.
\end{definition}

One method that has proved useful for studying recurrence relations, Christoffel--Darboux formulas and connection with Toda-type integrable hierarchies is the $LU$ factorization method, or also called the Gaussian factorization method \cite{cum} and \cite{afm-2}. The factorization problem consists in finding two matrices $S, \tilde S$ such that $g=S^{-1}\tilde S$. The matrix $S$ is a block lower triangular matrix with $I_{\N}$ in the diagonal, whereas the matrix $\tilde S$ must be block upper triangular as it is shown below
\begin{align}S&=\begin{pmatrix} I_N & 0_N & 0_N &\dots  \\
                                                 S_{1,0} & I_{N} & 0_{N}  & \dots \\
                                                 S_{2,0} & S_{2,1}& I_{N} & \dots \\
                                                 \vdots & \vdots & \vdots & \ddots \end{pmatrix},&
                                                   \tilde S&=\begin{pmatrix} \tilde S_{0,0} & \tilde S_{0,1} & \tilde S_{0,2} & \dots  \\
                                                 0_N & \tilde S_{1,1} & \tilde S_{1,2}  & \dots \\
                                                 0_N & 0_N & \tilde S_{2,2} & \dots \\
                                                 \vdots & \vdots & \vdots & \ddots \end{pmatrix}.
\end{align}

The coefficients of the matrices $S,\tilde S$ are the coefficients of the polynomials $\{p_n,\tilde p_n\}$ as we are going to prove now.
\begin{definition}
The semi-infinite block vectors $p$ and $\tilde p$ are given by
\begin{align}\begin{aligned}
p(x)&:=S\chi(x),\\
\tilde p(x)&:=(\tilde S ^{-1})^{\top}\chi(x). \end{aligned}
\end{align}
\end{definition}

Both $p=(p_0,p_1,\dots)^{\top}$ and $\tilde p=(\tilde p_0, \tilde p_1, \dots)^{\top}$ are sequences of matrix polynomials that, as we are going to see, are bi-orthogonal respect to the matrix weight $\rho(x)$.
\begin{pro}
Given the elements $p_i$ and $\tilde p_j$ of the sequences $p,\tilde p$ then $\langle p_i, p_j \rangle=0$.
\end{pro}
\begin{proof}
It is a consequence of the definition of $S$, $\tilde S$
\begin{align*}
\langle p_i, p_j \rangle&=\int_{\R} p_i(x)\rho(x) \tilde p_j(x)^{\top} \d x= \int_{\R}(p(x) \rho(x) \tilde p(x)^{\top})_{ij} \d x =(S\int_{\R}\chi(x) \rho(x) \chi(x)^{\top} \d x \tilde S^{-1})_{ij}\\
&=(Sg\tilde S^{-1})_{ij}=\I_{N}\delta_{ij}.
\end{align*}
\end{proof}

\subsubsection{Hankel and multigraded-Hankel symmetries}
\label{section.Hankel.multigraded}
Along with the Gaussian factorization, another important tool in our analysis is the Hankel-type symmetries that arise from the structure of the moment matrix. We can proceed with some definitions
\begin{definition}
\begin{enumerate}
\item
We say that a block-matrix $M_{ij}$ is a block-Hankel matrix if and only if $M_{i+1,j}=M_{i,j+1}$.
\item
The shift operator $\Lambda$ is a block semi-infinite matrix given by $\Lambda_{ij}=\I_{N}\delta_{i,j-1}$.
\end{enumerate}
\end{definition}
The matrix representation for the shift operator $\Lambda$ given in the definition has the following form
\begin{align*}
\Lambda &=\begin{pmatrix} 0_N & \I_N & 0_N & 0_N & \cdots \\ 0_N & 0_N & \I_N & 0_N & \cdots \\ 0_N & 0_N & 0_N & \I_N & \cdots \\ 0_N & 0_N & 0_N & 0_N & \cdots \\ \vdots & \vdots & \vdots & \vdots &\ddots \end{pmatrix} &
\Lambda^{\top} &=\begin{pmatrix} 0_N & 0_N & 0_N & 0_N & \cdots \\ \I_N & 0_N & 0_N & 0_N & \cdots \\ 0_N & \I_N & 0_N & 0_N & \cdots \\ 0_N & 0_N & \I_N & 0_N & \cdots \\ \vdots & \vdots & \vdots & \vdots &\ddots \end{pmatrix}.
\end{align*}

The so-called shift operator is useful not only for representing shifts on sequences, but also for characterizing the Hankel symmetry of the moment matrix $g$, as we will state in the following proposition that can be proved directly
\begin{pro}
\begin{enumerate}
\item
The shift operator $\Lambda$ and the sequence vector $\chi(x)$ have the following eigenvalue properties
\begin{align}
\Lambda \chi(x)&=x \chi(x) & \chi(x)^{\top} \Lambda^{\top}=x \chi(x)^{\top}.
\end{align}
\item
The moment matrix $g$ is a block-Hankel matrix and satisfies $\Lambda g = g \Lambda^{\top}$.
\end{enumerate}
\end{pro}

As we are interested in slightly more general situations the ordinary Hankel symmetry can be generalized to what we will call multigraded-Hankel symmetry as we see in the following definition. Given a multi-index $\vec n=(n_1,\dots,n_N)$ with $n_a$ non-negative integers we can define the following power $A^{\vec n}=\sum_{a=1}^N A^{n_a}E_{aa}$.
\begin{definition}
Given two multi-indices $\vec n$ and $\vec m$ we say that a block semi-infinite matrix $g$ is a multigraded-Hankel type matrix $(\vec n,\vec m)$ if
\begin{align}\label{mhankel}
    \Lambda^{\vec n} g= g (\Lambda^{\top})^{\vec m}.
\end{align}
\end{definition}
If we consider that $g_{ij}$ is a block of $g$ it is possible to extend the notation $g_{ij}=(g_{ij,ab})_{1\leq a,b\leq N}$. Using this notation the multigraded-Hankel condition can be written as $g_{i+n_a\,j,ab}=g_{i\, j+m_b,ab}$. It is possible to construct a family of matrices with this kind of symmetry using the following moment matrices
\begin{align}\label{mghankel}
  g_{ij,ab}=\int_\R x^i\rho_{j,ab}(x)\d x,
\end{align}
where the weights satisfy a general periodicity condition like
\begin{align}\label{periodicity}
  \rho_{j+m_b,ab}(x)=x^{n_a}\rho_{j,ab}(x).
\end{align}

Given the weights $\rho_{0,ab},\dots,\rho_{m_b-1,ab}$,
the rest of them are given by \eqref{periodicity}. This is a generalization of the moment problem studied by M. Adler and P. van Moerbeke in \cite{adler-van-moerbeke} for the particular case of one component ($N=1$) and with $n_1=m_1$, what they refered as the band condition. The consideration of the recursion relations that arise from this generalized symmetry can be also found in \cite{adler-vanmoerbeke-2}.



In this generalized case we have to define the adequate objects to preserve the similarities with the Hankel case.
\begin{definition}
In the multigraded-Hankel case we can define the following objects
\begin{enumerate}
\item
The vector sequences $\chi_1$ and $\chi_2$
\begin{align}
\chi_1(x)&:=(\I_N, x \I_N, x^2 \I_N, \dots )^{\top}, & \chi_2(x):=(\rho_0(x), \rho_1(x),\rho_2(x),\dots )^{\top},
\end{align}
where the weights verify the periodicity condition \eqref{periodicity}.
\item
The matrix moment defined by
\begin{align}
g:=\int_{\R}\chi_1(x) \chi_2(x)^{\top} \d x.
\end{align}
\item
The generalized families of matrix polynomials and matrix dual polynomials
\begin{align}
p(x)&:=S \chi_1(x), & \tilde p(x)&:=(\tilde S^{-1})^{\top}\chi_2(x).
\end{align}
\end{enumerate}
\end{definition}

The reader may notice that under this definition, the vector $p$ is still a family of matrix polynomials, but the dual family $\tilde p$ is not a family of polynomials but a family of linear combinations of matrix weights with coefficients given in $\tilde S^{-1}$. We will call this objects `` linear forms'' or ``generalized polynomials''. As we will see in the next section \ref{section.Matrix.CD} many of the formulas are valid in both cases.

The kind of equations that appear when we study this orthogonality problems suggest a connection with some kind of multiple orthogonality, as it is discussed in \cite{cum}. Now we are interested in the analogue of Christoffel--Darboux formulas that can be obtained if we adopt this generalized Hankel symmetry.

\section{Christoffel--Darboux formulas for generalized Hankel symmetries}
\label{section.Matrix.CD}

We will study some extensions of the classical Christoffel--Darboux formula for generalized Hankel symmetries. A known consequence of the Hankel symmetry that is present in the moment matrix associated to an orthogonality problem is the so called Chistoffel--Darboux formula. For a set $\{P_{n}\}$ of scalar orthogonal polynomials in the real line (under a weight $w(x)$) it is true that
\begin{align}\label{CD.classical}
\sum_{k=0}^{n-1}\frac{P_{k}(x)P_{k}(y)}{h_k}=\frac{1}{h_{n-1}}\frac{P_{n}(x)P_{n-1}(y)-P_{n-1}(x)P_{n}(y)}{x-y}.
\end{align}

Our objective in this section is to obtain a generalization of the classical formula \eqref{CD.classical} for the case of matrix polynomials and not standard Hankel symmetry.
\begin{definition}
Given a system of bi-orthogonal matrix polynomials $\{p_l,\tilde p_l\}$ we define the $l$-th Christoffel--Darboux Kernel as the following $N \times N$-matrix
\begin{align}
K^{[l]}(x,y):=\sum_{k=0}^{l-1} \tilde p_{k}(x)^{\top} p_{k}(y).
\end{align}
\end{definition}
One of its interesting properties has to do with the representation of orthogonal projections. It is obvious that any semi-infinite vector $v$ (with scalar or matrix components) can be written using a block notation as it is shown
\begin{align*}
v&= \left(\begin{array}{c}
v^{[l]}\\\hline\\[-10pt]
 v^{[\geq l]}
\end{array}\right)
\end{align*}
where $v^{[l]}$ is the semi-infinite vector with the first $l$ coefficients of $v$ and $v^{[\geq l]}$ the semi-infinite vector that contains the rest of the components.
This decomposition induces the following structure for an arbitrary semi-infinite matrix.
\begin{align*}
  g= \left(\begin{array}{c|c}
   g^{[l]}&g^{[l,\geq l]}\\\hline\\[-10pt]
     g^{[\geq l,l]}& g^{[\geq l]}
\end{array}\right).
\end{align*}
If we apply this to a $LU$ factorizable matrix $g$, the block structure is preserved by the factorization, that is:
\begin{align*}
  g^{[l]}&=(S^{[l]})^{-1}\tilde S^{[l]},&(S^{-1})^{[l]}&=(S^{[l]})^{-1},&(\tilde S^{-1})^{[\geq l]}&=(\tilde S^{[\geq l]})^{-1}.
\end{align*}

With this block structure it is natural to define the following $ \C^{N \times N} $ matrix polynomial spaces as the following
\begin{align}
\cH^{[l]}&:=\C\big\{\chi^{(0)},\dots,\chi^{(l-1)}\big\},
\end{align}
and its limit is the polynomial space
\begin{align}
\cH&:=\Big\{\sum_{l\leq k \ll  \infty} c_k \chi^{(k)},c_k\in\ C^{N \times N} \Big\},
\end{align}
where $l\ll \infty$ means that there are only a finite number of non zero terms.
That space $\cH$ has the following characterizations as well
\begin{align}
\cH^{[l]}&=\C\{p_0,\dots,p_{l-1}\}=\C\{\tilde p_0,\dots ,\tilde p_{l-1}\}.
\end{align}

We consider the following spaces
\begin{align*}
(\cH^{[l]})^{\bot_2}&:=\Big\{\sum_{l\leq k \ll  \infty} c_k p_k ,c_k\in\C^{N \times N}\Big\},& (\cH^{[l]})^{\bot_1}&:=\Big\{\sum_{l\leq k \ll  \infty} c_k \tilde p_k,c_k\in\C^{N \times N}\Big\},
\end{align*}
that have the following properties
\begin{align*}
\langle \cH^{[l]},(\cH^{[l]})^{\bot_1} \rangle &=0, & \langle(\cH^{[l]})^{\bot_2},\cH^{[l]}\rangle &=0,
\end{align*}
and give meaning to the following decompositions
\begin{align*}
 \cH &=\cH^{[l]}\oplus (\cH^{[l]})^{\bot_1}=\cH^{[l]}\oplus (\cH^{[l]})^{\bot_2},
\end{align*}
that induce the following projections
\begin{align}
  \pi_1^{(l)}&:\cH \to \cH^{[l]},&  \pi^{(l)}_2&:\cH \to \cH^{[l]},
\end{align}
where $ \pi_1^{(l)} $ is the projection onto $ \cH^{[l]} $ parallel to $ (\cH^{[l]})^{\bot_1} $ and $ \pi^{(l)}_2 $ is the projection onto $ \cH^{[l]} $ parallel to $ \oplus (\cH^{[l]})^{\bot_2} $. Those projections are extensions of the orthogonal projection to the bi-orthogonal case.

Now we make a comment on some remarkable and already known properties of the Christoffel--Darboux kernel.
\begin{pro}
\begin{enumerate}
\item The Christoffel--Darboux kernel is the integral kernel of the following projections
  \begin{align}
  \begin{aligned}
  (\pi^{(l)}_{ 1}f)(x)^{\top}&=\int_{\R} K^{[l]}(x,y) \rho(x) f(y)^{\top} \d y, & \forall f \in \cH,\\
  (\pi^{(l)}_{ 2}f)(x) &= \int_{\R}  f(y) \rho(x) K^{[l]}(y,x) \d y , &\forall f \in \cH.
  \end{aligned}
\end{align}
\item The kernel  $K^{[l]}(x,y)$ verifies the known reproducing property.
  \begin{align}
K^{[l]}(x,y)=\int_{\R} K^{[l]}(x,u) K^{[l]}(u,y) \d y.
\end{align}
\item $ K^{[l]}(x,y) $ is a sesquilinear form whose associated matrix is the inverse of the moment matrix (result known sometimes as Aitken--Berg--Collar theorem or  $ABC$ theorem \cite{Simon-CD}).
\begin{align}
     K^{[l]}(x,y)&=(\chi^{[l]}(x))^\dagger (g^{[l]})^{-1} \chi^{[l]}(y).
\end{align}
\end{enumerate}
\end{pro}

Now it is necessary to make a difference between the situation where the symmetry is pure Hankel and the Hankel-like generalized symmetry. The matrix moment in the Hankel case has been constructed using a matrix weight $\rho$ and the monomial sequence $ \chi$ as $g=\int_{\R} \chi(x) \rho(x)\chi(x)^{\top} \d x$. In this case the two families $p_l$ and $\tilde p_l$ are both families of matrix orthogonal polynomials. If the existent symmetry is more general (multigraded-Hankel) we have written the matrix $g$ as $g=\int_{\R} \chi_1(x) \chi_2(x)^{\top} \d x$. In this case the family $p_l$ (built using $\chi_1$) it still a family of polynomials, but the dual family $\tilde p_l$ (built using $\chi_2$) is no longer a family of polynomials but a family of what we call linear forms (linear combinations of weights). As it was studied in \cite{cum} and later in \cite{afm-2} the set $p_l$ can be understood as a family of multiple orthogonal polynomials of type II and the family $\tilde p_l$ can be understood as the set of the linear forms associated to the dual problem (that is a multiple orthogonal problem of type I). In this paper, the fact that $p_l$ and $\tilde p_l$ are in fact families of more general multiple orthogonal polynomials is not important and we will still use the ``polynomial'' language in a formal way.

The next step towards the obtention of the Christoffel--Darboux formula comes from this result
\begin{pro} \label{pro.matrix.CD}
 When the moment matrix $g$ has the multigraded-Hankel symmetry $\Lambda^{\vec n} g= g (\Lambda^{\vec m})^{\top}$ the kernel $K^{[l]}$ satisfies
\begin{align}
   \begin{aligned} x^{\vec n}K^{[l]}(x,y)-K^{[l]}(x,y)y^{\vec n}&=\Big(\chi_2^{[\geq l]}(x)^\top-\chi_2^{[l]}(x)^\top
   (g^{[l]})^{-1}  g^{[l,\geq l]}\Big) ((\Lambda^{\vec m})^{[l,\geq l]})^\top (g^{[l]})^{-1}
   \chi_1^{[l]}(y)\\&
   -\chi_2^{[l]}(x)^\top (g^{[l]})^{-1} (\Lambda^{\vec n})^{[l,\geq l]}\Big(\chi_1^{[\geq l]}(y)-g^{[\geq
   l,l]}(g^{[l]})^{-1}\chi_1^{[l]}(y)\Big).
     \end{aligned}
\end{align}
\end{pro}
\begin{proof}
The multigraded Hankel symmetry can be written using the block notation as follows
\begin{align*}
  (\Lambda^{\vec n})^{[l]}g^{[l]}+(\Lambda^{\vec n})^{[l,\geq l]}g^{[\geq l,l]}= g^{[l]}((\Lambda^{\vec m})^{[l]})^\top+g^{[l,\geq
  l]}((\Lambda^{\vec m})^{[l,\geq l]})^\top,
\end{align*}
or equivalently
\begin{align*}
(g^{[l]})^{-1}(\Lambda^{\vec n})^{[l]}-((\Lambda^{ \vec m})^{[l]})^\top (g^{[l]})^{-1}= (g^{[l]})^{-1} (g^{[l,\geq
  l]}((\Lambda^{ \vec m})^{[l,\geq l]})^\top -(\Lambda^{ \vec n})^{[l,\geq l]}g^{[\geq l,l]}) (g^{[l]})^{-1}.
\end{align*}

The eigenvalue property that has the shift operator $\Lambda$ can be extended to the multigraded case in the following manner. For the sequence of monomials $\chi_1$ it holds as usual that $\Lambda \chi_1(x)=x\chi_1(x)$, but for the sequence $\chi_2$ the component-like expression gives way to  $(\Lambda^{m_b}\chi_2(x))_{j,ab}=x^{n_a}(\chi_2(x))_{j,ab}$. This can be written in the following way using the following block expressions.
\begin{align*}
  (\Lambda^{ \vec n})^{[l]}\chi_1^{[l]}(x)&=x^{\vec n}\chi_1^{[l]}(x)-(\Lambda^{ \vec n})^{[l,\geq l]}\chi_1^{[\geq l]}(x), \\
  \chi_2^{[l]}(x)^{\top}((\Lambda^{ \vec m})^{[l]})^{\top}&=x^{\vec n}\chi_2^{[l]}(x)^{\top}-\chi_2^{[\geq l]}(x)^{\top}((\Lambda^{\vec m})^{[l,\geq l]})^{\top},
\end{align*}
from where the result can be obtained using the ABC  theorem, the eigenvalue property and the Hankel symmetry.
\end{proof}

If we want to simplify the obtained expression it is necessary to introduce some new elements. First, the orthogonality relations are equations of a linear system that can be written using a matrix formalism.
\begin{pro}
The family $p_l$ of orthogonal polynomials can be written using the following matrix notation
\begin{align}
\begin{aligned}
p_l(x)&=\chi_1^{(l)}(x)-\begin{pmatrix} g_{l,0}& g_{l,1}& \cdots & g_{l,l-1} \end{pmatrix} (g^{[l]})^{-1}\chi_1^{[l]}(x) \\
      &=\tilde S_{ll} \begin{pmatrix} 0_N &0_N & \cdots & \I_N \end{pmatrix} (g^{[l+1]})^{-1} \chi_1^{[l+1]}(x),
\end{aligned}
\end{align}
and the (generalized) dual polynomials $\tilde p_l$ can be constructed using the following notation
\begin{align}
\begin{aligned}
\tilde p_l(x)^{\top}&=\Big(\chi^{(l)}(x)^{\top}-\chi_2^{[l]}(x)(g^{[l]})^{-1} \begin{pmatrix} g_{0,l} & g_{1,l} & \cdots & g_{l-1,l} \end{pmatrix}^{\top} \Big) \tilde S_{ll}^{-1} \\
     &= \chi_2^{[l+1]}(x)^{\top} (g^{[l+1]})^{-1} \begin{pmatrix} 0_N & 0_N & \cdots & \I_N \end{pmatrix}^{\top}.
\end{aligned}
\end{align}
\end{pro}
As a second step it is necessary to define new families of (generalized) polynomials that we will call associated polynomials
\begin{definition} \label{def.matrix.associated}
We define the following associated polynomials
\begin{align}
\begin{aligned}
p_{l,+j}(x)&:=\chi_1^{(l+j)}(x)-\begin{pmatrix}g_{l+j,0} & g_{l+j,1} & \cdots & g_{l+j,l-1} \end{pmatrix} (g^{[l]})^{-1}\chi_1^{[l]}(x), \\
p_{l,-j}(x)&:= e_{l-j}^{\top} (g^{[l+1]})^{-1}
   \chi_1^{[l+1]}(x),
\end{aligned}
\end{align}
and the dual associated polynomials
\begin{align}
\begin{aligned}
\tilde p_{l,+j}(x)^{\top}&:=\chi_2^{(l+j)}(x)^\top-\chi_2^{[l]}(x)^\top
   (g^{[l]})^{-1}  \begin{pmatrix} g_{0,l+j} & g_{1,l+j} & \cdots & g_{l-1,l+j} \end{pmatrix}^{\top},\\
\tilde p_{l,-j}(x)^{\top}&:=\chi_2^{[l+1]}(x)^\top (g^{[l+1]})^{-1}e_{l-j},
\end{aligned}
\end{align}
where\footnote{The space dimension is not explicitly written. That is a little abuse of notation.} $e_j= (\underbrace{\begin{matrix} 0_N & 0_N & \cdots & 0_N \end{matrix}}_{j} \begin{matrix} & \I_N & 0_N & \cdots & \end{matrix})^{\top}$.
\end{definition}
From the formulas we can see that
\begin{align}
\begin{aligned}
p_l&=p_{l,+0}=\tilde S_{ll} p_{l,-0}, \\
\tilde p_l^{\top}&=\tilde p_{l,+0}^{\top}\tilde S_{ll}^{-1}=\tilde p_{l,-0}^{\top}.
\end{aligned}
\end{align}

It is also interesting to notice that the new polynomials satisfy modified orthogonality relations that can be easily expressed using the non-generalized polynomials. We summarize all those results in the following proposition.
\begin{pro}
\begin{enumerate}
\item The polynomials  $p_{l,+j}$ y $ \tilde p_{l,+j}$ are monic polynomials of degree $l+j$ that satisfy
   \begin{align}
   \begin{aligned}
   \int_{\R} p_{l,+j}(x) \rho_k (x) \d x &= 0_{N}, k=0,\dots,l-1, \\
   \int_{\R}  x^k \tilde p_{l,+j}(x)^{\top} \d x &= 0_{N}, k=0,\dots,l-1,
   \end{aligned}
   \end{align}
   and both $p_{l,-j}$ and $\tilde p_{l,-j}$ are matrix polynomials of degree $l$ that satisfy generalized orthogonality relations
   \begin{align}
   \begin{aligned}
   \int_{\R} p_{l,-j} \rho_k (x) \d x &= \delta_{k,l-j} \I_{N}, \\
   \int_{\R} x^k \tilde p_{l,-j}^{\top} \d x &= \delta_{k,l-j} \I_{N}.
   \end{aligned}
   \end{align}
\item
  The expressions that connect regular and associated polynomials are the following
  \begin{align}
  \begin{aligned}
  p_{l,+j}& = p_{l+j}+ S_{l+j,l+j-1}^{-1}p_{l+j-1}+ \cdots + S_{l+j,l}^{-1}p_{l}, \\
  p_{l,-j}& = \tilde S_{l-j,l}^{-1} p_{l}+ \tilde S_{l-j,l-1}^{-1} p_{l-1}+ \cdots + \tilde S_{l-j,l-j}^{-1} p_{l-j},
  \end{aligned}
  \end{align}
  associated dual polynomials satisfy
  \begin{align}
\begin{aligned}
\tilde p_{l,+j}&=\tilde S_{l+j,l+j}^{\top}\tilde p_{l+j}+ \tilde S_{l+j-1,l+j}^{\top} \tilde p_{l+j-1}+\cdots+\tilde S_{l,l+j}^{\top} \tilde p_{l},\\
\tilde p_{l,-j}&=S_{l,l-j}^{\top} \tilde p_{l} +S_{l-1,l-j}^{\top} \tilde p_{l-1} + \cdots + p_{l-j}.
\end{aligned}
\end{align}
\end{enumerate}
\end{pro}
\begin{proof}
\begin{enumerate}
\item Are a direct consequence of the definition. It is just necessary to multiply by the right by $(\chi_2^{[l]})^{\top}$ (for the non-duals) or by the left by $\chi_1^{[l]}$ (for the duals).
\item
 For the polynomials $p_{l,+j}$ the orthogonality relations that they verify make that \\ $p_{l,+j} \in \text{span} \{p_{l},p_{l+1},\ldots,p_{l+j}\} $ and we can thus write as a consequence $p_{l,+j}=a_{j}p_{l+j}+a_{j-1}p_{l+j-1}+\cdots+a_{0}p_{l}$. The coefficients also satisfy the following system of linear equations.
\begin{align*}
\begin{pmatrix} 1 & 0 & \cdots & 0 \end{pmatrix} = \begin{pmatrix} a_j & a_{j-1} & \cdots & a_0 \end{pmatrix} \begin{pmatrix} 1 & S_{l+j,l+j-1} & \cdots & S_{l+j,l} \\ 0 & 1 & \cdots & S_{l+j-1,l} \\ \vdots & \vdots & & \vdots \\ 0 & 0 & \cdots & 1 \end{pmatrix}
\end{align*}
from where we obtain easily that $ p_{l,+j}= p_{l+j}+ S_{l+j,l+j-1}^{-1}p_{l+j-1}+ \cdots + S_{l+j,l}^{-1}p_{l} $. For the family of associated polynomials $p_{l,-j}$, as $p_{l,-j}$ is a polynomial of degree $l$, we have that $ p_{l,-j} \in \text{span} \{p_{l},p_{l-1}, \dots, p_0\}$. Using again the orthogonality relations that they satisfy we obtain thus the following formulas for the coefficients.
\begin{align*}
\int p_{l,-j}(x) \tilde p_{k}(x)^{\dagger} \d x &= \tilde S_{lk}^{-1}, & k&=0,1 \dots l,
\end{align*}
 that is, $p_{l,-j}=\tilde S_{l-j,l}^{-1} p_{l}+ \tilde S_{l-j,l-1}^{-1} p_{l-1}+ \cdots + \tilde S_{l-j,l-j}^{-1} p_{l-j}$. The same reasoning can be applied to the dual associated (generalized) polynomials.
\end{enumerate}
\end{proof}

We are now prepared to obtain the following theorem
\begin{theorem}
For $l \geq \max\{|\vec n|,|\vec m|\}$ we have the following Christoffel--Darboux formula
\begin{align}
(x^{\vec n}K^{[l]}-K^{[l]}y^{\vec n})(x,y)=\sum_{a=1}^{N}\Big(\sum_{j=0}^{m-1} \tilde p_{l,+j}(x)^{\top} E_{aa} p_{l-1,-(m_a-j-1)}(y)
-\sum_{j=0}^{n-1} \tilde p_{l-1,-(n_a-j-1)} (x)^{\top} E_{aa} p_{l,+j}(y)\Big).
\end{align}
\end{theorem}
\begin{proof}
Using direct calculus it can be obtained that $ (\Lambda^n)^{[l,\geq l]}=\sum_{j=0}^{n-1} e_{l-n+j}e_{j}^\top $ for $l \geq n $, so that making substitutions in the proposition \ref{pro.matrix.CD} we have that for $l \geq \max \{|\vec n|,|\vec m|\}$
\begin{align*}
  (x^{\vec n} K^{[l]}(x,y) - K^{[l]}(x,y) y^{\vec n})&=\sum_{a=1}^N \sum_{j=0}^{m_a-1}\Big(\chi^{[\geq l]}(x)^\top-\chi^{[l]}(x)^\top
   (g^{[l]})^{-1}  g^{[l,\geq l]}\Big)e_j E_{aa} e_{l-m_a+j}^{\top} (g^{[l]})^{-1}
   \chi^{[l]}(y)\\&
   -\sum_{j=0}^{n_a-1}\chi^{[l]}(x)^\top (g^{[l]})^{-1}e_{l-n+j} E_{aa} e_{j}^{\top} \Big(\chi^{[\geq l]}(y)-g^{[\geq
   l,l]}(g^{[l]})^{-1}\chi^{[l]}(y)\Big).
\end{align*}
If we use the definition \ref{def.matrix.associated} for the associated polynomials we obtain the formula that we were looking for.
\end{proof}

\begin{cor}
The components of the matrix $K^{[l]}$ that we will call $K^{[l]}_{a,b}$ verify the following scalar equation for $l \geq \max\{|\vec n|,|\vec m|\}$ y $a,b=1,\dots,N$.
\begin{align}
K^{[l]}(x,y)_{ab}=\sum_{c=1}^{N}\frac{\sum_{j=0}^{m-1} \tilde p_{l,+j}(x)^{\top}_{ac}  p_{l-1,-(m_c-j-1)}(y)_{cb}
-\sum_{j=0}^{n-1} \tilde p_{l-1,-(n_c-j-1)} (x)^{\top}_{ac} p_{l,+j}(y)_{cb}}{x^{n_a}-y^{n_b}}.
\end{align}
\end{cor}

\end{document}